\documentclass[12pt]{article}

\usepackage{amsmath,amsfonts,amssymb,amsthm,amscd}

\title{Some model theory of $\SL(2,\R)$}

\date{\today}

\author{Jakub Gismatullin\thanks{Supported by the Marie Curie Intra-European Fellowship MODGROUP no. PIEF-GA-2009-254123 and Polish Goverment MNiSW grant N N201 545938}\\University of Leeds and Uniwersytet Wroc\l awski
\and 
Davide Penazzi\thanks{Supported by a postdoctoral fellowship on EPSRC grant EP/I002294/1}\\University of Leeds
\and 
Anand Pillay\thanks{Supported by EPSRC grant EP/I002294/1}\\University of Leeds}

\newtheorem{Theorem}{Theorem}[section]
\newtheorem{Proposition}[Theorem]{Proposition}
\newtheorem{Definition}[Theorem]{Definition} 
\newtheorem{Remark}[Theorem]{Remark}
\newtheorem{Lemma}[Theorem]{Lemma}
\newtheorem{Corollary}[Theorem]{Corollary}
\newtheorem{Fact}[Theorem]{Fact}

\newtheorem*{Claim}{Claim}

\newcommand{\R}{\mathbb R}

\newcommand{\Z}{\mathbb Z}

\newcommand{\Pp}{\mathbb P}

\DeclareMathOperator{\SL}{SL}
\DeclareMathOperator{\PSL}{PSL}
\DeclareMathOperator{\SO}{SO}

\DeclareMathOperator{\Th}{Th}
\DeclareMathOperator{\tp}{tp}
\DeclareMathOperator{\Fin}{Fin}
\DeclareMathOperator{\dcl}{dcl}
\DeclareMathOperator{\st}{st}
\DeclareMathOperator{\cl}{cl}

\begin{document}

\maketitle

\begin{abstract} 
We study the action of $G = \SL(2,\R)$, viewed as a group definable in the structure $M = (\R,+,\times)$, on its type space $S_{G}(M)$. We identify a minimal closed $G$-flow $I$, and an idempotent $r\in I$ (with respect to the Ellis semigroup structure $*$ on $S_{G}(M)$). 
We also show that the ``ideal group" $(r*I,*)$ is nontrivial (in fact it will be the group with $2$ elements), yielding a negative answer to a question of Newelski.
\end{abstract}

\section{Introduction and preliminaries}

Abstract topological dynamics concerns the actions of  (often discrete) groups $G$ on compact Hausdorff spaces $X$. Newelski has suggested in a number of papers such as \cite{Newelski1}, \cite{Newelski2}, that the notions of topological dynamics may be useful for ``generalized stable group theory", namely the understanding of definable groups in unstable settings, but informed by methods of stable group theory. Given a structure $M$ and a group $G$ definable in $M$, we have the (left) action of $G$ on its type space $S_{G}(M)$. When $\Th(M)$ is stable, there is a unique minimal closed $G$-invariant subset $I$ of $S_{G}(M)$ which is precisely the set of generic types of $G$. Moreover (still in the stable case) $S_{G}(M)$ is equipped with a semigroup structure $*$: $p*q = \tp(a\cdot b/M)$ where $a, b$ are independent realizations of $p,q$ respectively, and $(I,*)$ is a compact Hausdorff topological group which turns out to be isomorphic to $G\left(\overline{M}\right)/ G\left(\overline{M}\right)^{0}$ where $\overline{M}$ is a saturated elementary extension of $M$. In fact this nice situation is more or less characteristic of the stable case, so will not extend as such to unstable settings (other than what we have called in \cite{NIPII} ``generically stable groups").  However it was shown in \cite{Pillay-fsg} that for the much larger class of so-called $fsg$ groups definable in $NIP$ theories, the situation is not so far from in the stable case. In the $o$-minimal context the $fsg$ groups are precisely the definably compact groups; for example working in the structure $(\R,+,\times)$, these will be the semialgebraic compact Lie groups. However there is no general model-theoretic machinery (of a stability-theoretic nature) for understanding {\em simple non compact real Lie groups} (and their interpretations in arbitrary real closed fields). In this paper we try to initiate such a study, focusing on $G = \SL(2,\R)$. The reason we work over the standard model $(\R,+,\times)$ rather than an arbitrary or saturated model is that all types over the standard model are definable, hence externally definable sets correspond to definable sets, and the type space is equipped with an ``Ellis semigroup structure" $*$. We expect that analogues of our results hold over arbitrary models, expanded by the externally definable sets. In any case the main objective is to identify a minimal closed $G$-invariant subset $I$ of $S_{G}(M)$, to identify an idempotent element $r\in I$ and to describe the ``ideal group" $r*I$. Now $r*I$ as an abstract group does not depend on the choice of $I$ or $r$.  Newelski asked in \cite{Newelski1} whether  for groups $G$ definable in $NIP$ theories, $G\left(\overline{M}\right)/G\left(\overline{M}\right)^{00}$ is isomorphic, as an abstract group to this $r*I$. In \cite{Pillay-fsg} we gave a positive answer for $fsg$ groups in $NIP$ theories. When $G = \SL(2)$ and $K$ is a saturated real closed field then $G(K)$ is simple (modulo its finite centre) as an abstract group, whereby $G(K) = G(K)^{00}$. However we will show that in the case of $\SL(2,\R)$ acting on its type space the ideal group $r*I$ is the group with $2$ elements, in particular nontrivial, so giving a negative answer to Newelski's question.
Our idempotent will be obtained as an ``independent" (with respect to forking) product of realizations of a generic type of $T^{00}$ over $\R$ and an $H(\R)$-invariant type of $H$ where $T$ is a maximal compact and $H$ is the standard Borel subgroups of $\SL(2,\R)$. Moreover $I$ will be a ``universal minimal" $G$-flow, from the point of view of ``tame" topological dynamics, discussed briefly in the next paragraph. \\

So before getting on to more detailed preliminaries, let us mention that the third author has recently developed  \cite{Pillay-tame} a theory of ``tame" or ``definable" topological dynamics, concerning roughly the action of $(G,{\cal B})$ on a compact space $X$, where $G$ is a (discrete) group, and ${\cal B}$ is a $G$-invariant Boolean algebra of subsets of $G$. We will not give the definition here, but when ${\cal B}$ is the Boolean algebra of all subsets of $G$, this notion  reduces to the standard notion of the (discrete) group $G$ acting by homeomorphisms on $X$. When $G$ is a group definable in a structure $M$, and ${\cal B}$ is the Boolean algebra of definable (with parameters) subsets of $G$, then the ``universal $(G,{\cal B})$-flow with a distinguished dense orbit" is the type space $S_{G}(M)$  (rather than the Stone-Cech compactification $\beta G$ of $G$ as in the standard case). So this makes the study of the action of $G$ on its type space more attractive or intrinsic, although we will not need to know anything about this general theory for the purposes of the current paper. \\

We will assume a basic knowledge of model theory (types, saturation, definable types, heirs, coheirs,....). References are \cite{Poizat} and \cite{Pillay-book}. Let us fix a complete $1$-sorted theory $T$, a saturated model $\overline{M}$ of $T$, and a model $M$ which is elementary substructure of $\overline{M}$. In the body of the paper $T$ will be $RCF$, the theory of real closed fields, in the language of rings, and $M$ will be the ``standard model" $(\R,+,\times, 0, 1)$. By a definable set in $M$ we mean a subset of $M^{n}$ definable (with parameters) in $M$, namely by a formula $\phi(x_{1},..,x_{n},{\bar b})$ where we exhibit the parameters ${\bar b}$ from $M$. $S_{n}(M)$ is the space of complete $n$-types over $M$, equivalently, ultrafilters on the Boolean algebra of definable subsets of $M^{n}$ (which we identify with the Boolean algebra of formulas $\phi(x_{1},..,x_{n})$ with parameters from $M$, up to equivalence).  This is a compact Hausdorff space, under the Stone space topology.

\begin{Definition}
\begin{enumerate}
\item[(i)] A subset $X\subseteq M^{n}$ is externally definable if there is a formula $\phi(x_{1},..,x_{n},{\bar b})$ where the parameters ${\bar b}$ are from $\overline{M}$, such that $X = \{{\bar a}\in M^{n}: \overline{M}\models \phi({\bar a},{\bar b})\}$.
\item[(ii)] By $S_{ext, n}(M)$ we mean the space of ultrafilters on the Boolean algebra of externally definable subsets of $M^{n}$.
\end{enumerate}
\end{Definition}

\begin{Fact} \label{fact:hom} 
For any $p({\bar x})\in S_{ext,n}(M)$, there is a unique $p'({\bar x})\in S_{n}\left(\overline{M}\right)$ which is finitely satisfiable in $M$ and such that the ``trace on $M$" of any formula in $p'$ is in $p$. This sets up a homeomorphism $\iota$ between $S_{ext,n}(M)$ and the closed subspace of $S_{n}\left(\overline{M}\right)$ consisting of types finitely satisfiable in $M$.
\end{Fact}

Note that if all types over $M$ are definable, then externally definable subsets of $M^{n}$ are definable and $S_{ext,n}(M)$ coincides with $S_{n}(M)$.

\begin{Lemma} \label{lem:heir} Suppose that all types over $M$ are definable, and let $p(x)\in S_{n}(M)$.
\begin{enumerate}
\item[(i)] For any $B\supseteq M$, $p$ has a unique coheir $p'(x)\in S_{n}(B)$, namely an extension of $p$ to a complete type over $B$ which is finitely satisfiable in $M$.
\item[(ii)] For any $B\supseteq M$, $p$ has a unique heir over $B$, which we write as $p|B$. $p|B$ can also be characterized as the unique extension of $p$ to $B$ which is definable over $M$. Moreover $p|B$ is simply the result of applying the defining schema for $p$ to the set of parameters $B$.
\item[(iii)] For any tuples $b,c$ from $\overline{M}$, $\tp(b/M,c)$ is definable over $M$ if and only if $\tp(c/M,b)$ is finitely satisfiable in $M$.
\end{enumerate}
\end{Lemma}

Now suppose $G$ is a group definable over $M$. We identify $G$ with the group $G\left(\overline{M}\right)$ and write $G(M)$ for the points in the model $M$. We have the spaces of types $S_{G}(M)$, $S_{ext,G}(M)$ and $S_{G}\left(\overline{M}\right)$. For $g,h\in G$ we write $gh$ for the product. $G(M)$ acts (on the left) by homeomorphisms on $S_{G}(M)$ and $S_{ext,G}(M)$. 

\begin{Definition} Let $p(x), q(x)\in S_{ext,G}(M)$. Let $b$ realize $q$ in $G$, and let $a$ realize the unique $p'\in S_{G}\left(\overline{M}\right)$ given by \ref{fact:hom}. We define $p*q$ to be the (external) type of $ab$ over $M$. So in the case when all types over $M$ are definable, this just means: let $b\in G$ realize $q$ and let $a\in G$ realize the unique coheir of $p$ over $M,b$, then $p*q = \tp(ab/M)$.
\end{Definition}

The following is contained in \cite{Newelski1} and \cite{Newelski2}. Everything can be proved directly, but it is a special case of the theory of abstract topological dynamics, as treated in \cite{Auslander} for example.

\begin{Lemma}
\begin{enumerate}
\item[(i)] $(S_{ext,G}(M),*)$ is a semigroup, which we call the Ellis semigroup, and $*$ is continuous in the first coordinate, namely for any $q\in S_{ext,G}(M)$ the map taking $p\in S_{ext,G}(M)$ to $p*q\in S_{ext,G}(M)$ is continuous.
\item[(ii)] Left ideals of $S_{ext,G}(M)$ (with respect to $*$) coincide with subflows, namely closed $G(M)$-invariant subsets.
\item[(iii)] If $I\subseteq S_{ext,G}(M)$ is a minimal subflow, then $I$ contains an idempotent $r$ such that $r*r = r$, and $(r*I,*)$ is a group, whose isomorphism type does not depend on $I$ or $r$.
\end{enumerate}
\end{Lemma}

As mentioned earlier, in the stable case there is a unique minimal subflow, the space of generic types of $G$ over $M$. We will, below, consider the case where $T$ is the theory of real closed fields, $M = (\R,+,\times)$ is the standard model and $G = \SL(2,-)$. Sometimes we write $\R$ for $M$ to be consistent with standard notation. So $G(\R)$ is the interpretation of $G$ in $M$, namely $\SL(2,\R)$, and $\R$ as a structure is $(\R,+,\times)$. It is well-known that all types over $\R$ are definable \cite{Marker-Steinhorn}, hence Lemma 1.5 applies to $G(\R)$ acting on $S_{G}(\R)$.

\section{$\SL(2,\R)$} \label{sec:sl}

We review some basic and well-known facts about $\SL(2,\R)$, the group of $2\times 2$ matrices over $\R$ with determinant $1$. All the objects, maps etc. we mention will be semialgebraic and so pass over to $\SL(2,K)$ where $K$ is a saturated real closed field. We sometimes write $G$ for $\SL(2)$, so $G(\R)$ for $\SL(2,\R)$. Write $I$ for the identity matrix. The centre of $\SL(2,\R)$ is $\{I,-I\}$. The quotient of $\SL(2,\R)$ by this centre is called $\PSL(2,\R)$. \\


$H(\R)$ will denote the standard Borel subgroup of $G(\R)$, namely the subgroup consisting of matrices $\left(\begin{array}{cc} b& c\\ 0&b^{-1}    \end{array}\right)$ where $b\in \R_{>0}$ and $c\in \R$. $H(\R)$ is precisely the semidirect product of $\left(\R_{>0},\times\right)$ with $(\R,+)$. We let $T(\R)$ denote $\SO(2,\R)$: the subgroup of $G(\R)$ consisting of matrices $\left(\begin{array}{cc} x&-y\\ y&x    \end{array}\right)$  with $x,y\in \R$ and $x^{2} + y^{2} = 1$. The symbol $T$ here stands for torus. $H(\R)\cap T(\R) = \{I\}$ and any element of $G$ can be written uniquely in the form $ht$ (as well as $t_{1}h_{1}$) for $t,t_{1}\in T$ and $h,h_{1}\in H$. $T(\R)$ is a maximal compact subgroup of $G(\R)$. Note that $-I \in T(\R)$. \\
  
We write $V(\R)$ for the homogeneous space $G(\R)/H(\R)$  (space of left cosets $\{gH(\R): g\in G(\R)\}$), and $\pi$ (or $\pi(\R)$) for the projection  $G(\R) \to V(\R)$. Note that $\pi_{|T(\R)}\colon T(\R) \to V(\R)$ is a homeomorphism. We write the action of $G(\R)$ on $V(\R)$ by $\cdot$. Understanding this action will be quite important for us. The usual action of $G(\R)$ on the real projective line by Mobius transformations factors through the action of $G(\R)$ on $V(\R)$, and we will try to describe what is going on.\\
  
So this standard action of $G(\R)$ on $\Pp^{1}(\R)$ is: 
$\left(\begin{array}{cc} a& b\\ c&d    \end{array}\right)\cdot \left(\begin{array}{c} x\\y    \end{array}\right) = \left(\begin{array}{c} ax+by\\ cx+dy    \end{array}\right)$, where  $\left(\begin{array}{c} x\\ y    \end{array}\right)$ is a representative of an element of $\Pp^{1}(\R)$. It is well defined because $\left(\begin{array}{cc} a& b\\ c&d    \end{array}\right)$ has determinant $1$. 

We identify  $\left(\begin{array}{c} x\\ 1    \end{array}\right)$ with $x\in \R$, and treat $\left(\begin{array}{c} 1\\ 0    \end{array}\right)$ as the ``point at infinity". It is easy to prove the following fact.

\begin{Remark} \label{rem:stab}
\begin{enumerate}
\item[(i)] $Stab_{G(\R)}\left(\begin{array}{c} 1\\ 0    \end{array}\right) =  H_{1}(\R)$, where $H_{1}(\R) = H(\R)\times \{I,-I\}$.
\item[(ii)] $Z(G(\R)) = \{I,-I\}$ acts trivially on $\Pp^{1}(\R)$, and the resulting action of $\PSL(2,\R) = G(\R)/Z(G(\R))$ on $\Pp^{1}(\R)$ is the usual faithful action.
\end{enumerate}
\end{Remark}

Let $\pi_{1}$ denote the map from $G(\R)$ to $\Pp^{1}(\R)$ taking $g$ to $g\cdot \left(\begin{array}{c} 1\\ 0    \end{array}\right)$.  
So by Remark \ref{rem:stab}$(i)$, $\pi_{1}$ induces an isomorphism of $G(\R)$-homogeneous spaces $G(\R)/H_{1}(\R)$ and $\Pp^{1}(\R)$. Moreover 

\begin{Remark} \label{rem:homeo}
The restriction of $\pi_{1}$ to $T(\R)$ induces a homeomorphism between $T(\R)/\{I,-I\}$ and  $\Pp^{1}(\R)$, such that the identity of $T(\R)/\{I,-I\}$ goes to $\left(\begin{array}{c} 1\\ 0    \end{array}\right)$. 
\end{Remark}

Finally, by virtue of the homeomorphism $\pi_{|T(\R)}$ between $T(\R)$ and $V(\R)$ and the action of $G(\R)$ on $V(\R)$, we have an action (also written $\cdot$) of $G(\R)$ on $T(\R)$. Note  that $g\cdot t$ is the unique $t_{1}\in T(\R)$ such that $gt = t_{1}h_{1}$ for some (unique) $h_{1}\in H(\R)$. 
Likewise by virtue of Remark \ref{rem:homeo}, and the action of $G(\R)$ on $\Pp^{1}(\R)$,  we obtain an action $\cdot_{1}$ of $G(\R)$ on $T(\R)/\{I,-I\}$. We clearly have:

\begin{Remark} \label{rem:act} The action $\cdot_{1}$ of $G(\R)$ on $T(\R)/\{I,-I\}$ is induced by the action $\cdot$ of $G(\R)$ on $T(\R)$. In particular, for any $g\in G(\R)$ and $t\in T(\R)$, $g\cdot t  \in g\cdot_{1}(t/\{I,-I\})$. 
\end{Remark}

As remarked above all this passes to a saturated model $K$ of $RCF$ in place $\R$. We write $G$ for $G(K) = \SL(2,K)$, $H$ for $H(K)$ etc, $V$ for $V(K)$ etc. But now our groups and homogeneous spaces contain nonstandard points, and the study of their types and interaction, is what this paper is about.

\section{Main results}

We follow the conventions at the end of the last section. ($G = \SL(2)$, $K$ a saturated real closed field, etc.) We say that $a\in K$ is \emph{infinite}, if $a>\R$. And call $a$ \emph{negative infinite} if $a< \R$. $\Fin(K)$ denotes the elements of $K$ which are neither infinite nor negative infinite. Any $a\in \Fin(K)$ has a standard part $\st(a)\in \R$.  Also given $B\subset K$, $a$ is \emph{infinite (negative infinite) over $B$} if $a > \dcl(B)$ ($a<\dcl(B)$). Call $a\in K$ \emph{positive infinitesimal} if $a > 0$ and $a<r$ for all positive $r\in \R$. Likewise for \emph{negative infinitesimal} and for \emph{infinitesimal over $B$}. Note that if for example $a\in K$ is positive infinitesimal, $p(x)= \tp(a/\R)$, and $B\subset K$ then $p|B$ is the type of an element which is positive infinitesimal over $B$.

We sometimes write $g/H$ for the left  coset $gH$. The projection $\pi\colon G \to V = G/H$ induces a surjective continuous map which we also call $\pi$ from $S_{G}(\R)$ to $S_{V}(\R)$. Both these type spaces are acted on (by homeomorphisms) by $G(\R)$, and we clearly have:

\begin{Lemma} \label{lem:inv} $\pi$ is $G(\R)$-invariant: namely for any $p\in S_{G}(\R)$, and $g\in G(\R)$, $\pi(gp) = g\cdot \pi(p)$.
\end{Lemma}

\begin{Definition} Let $p_{1}\in S_{G}(\R)$, and $q\in S_{V}(\R)$. Define $p_{1}*q$ to be $\tp(g\cdot b)$ where $b$ realizes $q$, and $g$ realizes the unique coheir of $p_{1}$ over $M,b$. 
\end{Definition}

With above notation, the following extends Lemma \ref{lem:inv}.

\begin{Lemma} \label{lem:p} For any $p, p_{1}\in S_{G}(\R)$, $\pi(p_{1}*p) = p_{1}*\pi(p)$. 
\end{Lemma}
\begin{proof} Fix $p, p_{1}\in G$. Then $p_{1}*p = \tp(g_{1}g/\R)$ where $g_{1}$ realizes $p_{1}$, $g$ realizes $p$ and $\tp(g_{1}/\R,g)$ is finitely satisfiable in $\R$. But then $\pi(p_{1}*p) = \tp((g_{1}g/H)/\R) = \tp(g_{1}\cdot(g/H)/\R)$.
Now $\tp(g_{1}/\R, g/H)$ is finitely satisfiable in $\R$ and $g/H$ realizes $\pi(p)$. Hence $\tp(g_{1}\cdot(g/H)/\R) = p_{1}*\pi(p)$, as required.
\end{proof}


As above the symbol $g$ will range over elements of $G$. Also $h$ ranges over elements of $H$, and $t$ over elements of $T$. If $h = \left(\begin{array}{cc} b& c\\ 0&b^{-1}    \end{array}\right)$ is in $h$ we identify it with the pair  $(b,c) \in \R_{>0}\times \R$. And if $t = \left(\begin{array}{cc} x&-y\\ y&x    \end{array}\right)$ is an element of $T$ we identify it with the pair $(x,y)$  (so $T$ is identified with the unit circle under complex multiplication).\\

We now fix some canonical types: $p_{0} = \tp(b,c/\R)$ where $b$ is infinite and $c$ is infinite over $b$. It is easy to check that $p_{0}$ is \emph{left $H(\R)$-invariant}. Namely if $h$ realizes $p_{0}$ and $h_{1}\in H(\R)$, then $h_{1}h$ also realizes $p_{0}$.\\

Note that all nonalgebraic types (over $\R$) of elements of $T$ are generic in the sense of \cite{NIPII}. In fact $T$ is the simplest possible $fsg$ group in $RCF$. Let $q_{0} = \tp(x,y/\R)$  (as the type of an element of $T$) where $y$ is positive infinitesimal and $x>0$ (so $x$ is the positive square root of $1-y^{2}$). We call $q_{0}$ the type of a ``positive infinitesimal" of $T$: it is infinitesimally close to the identity, on the ``positive" side.  \\

Likewise, for any $t\in T(\R)$ and $t_{1}\in T$, we will say that $t_{1}$ is ``infinitesimally close, on the positive side" to $t$ if $t_{1}t^{-1}$ realizes $q_{0}$. \\

The bijection (homeomorphism) between $T$ and $V$ given by $\pi_{|T}$ induces a homeomorphism (still called $\pi$) between $S_{T}(\R)$ and $S_{V}(\R)$, so we will sometimes identify them below, although we distinguish between $q$ and $\pi(q)$ (for $q\in S_{T}(\R)$).


\begin{Definition} \label{def:type}  We define $r_{0}$ to be $\tp(th/\R)$ where $h\in H$ realizes $p_{0}$ and $t\in T$ realizes the unique coheir of $q_{0}$ over $\R,h$.
\end{Definition}

Note that $\pi(r_{0}) = \pi(q_{0})$. Our first aim is to show that $\cl(G(\R)r_{0}) = I$ is a minimal closed $G(\R)$-flow, and that $r_{0}$ is an idempotent. Note that $\cl(G(\R)r_{0})$ is precisely the set of $p*r_{0}\in S_{G}(\R)$ for $p$ ranging over $S_{G}(\R)$. Likewise for $\cl(G(\R)\cdot\pi(r_{0}))$.

\begin{Lemma} \label{lem:stab}
 For any $p\in S_{G}(\R)$, $p*\pi(q_{0}) = \pi(q_{0})$ if and only if $p$ is of the form $\tp(t_{1}h_{1}/\R)$ with $h_{1}\in H$ and $t_{1}\in T$ the identity or a realization of $q_{0}$.
\end{Lemma}
\begin{proof} 
Let $\tp(t_{1},h_{1}/\R,t)$ be finitely satisfiable in $\R$ with $h_{1}\in H$, $t_{1}\in T$ and $t$ realizing $q_{0}$  (so $t/H$ realizes $\pi(q_{0})$).   $\tp(t/\R,t_{1},h_{1})$ is the unique heir of $q_{0}$ over $(\R,t_{1},h_{1})$. In particular $t\in T$ is positive infinitesimal over $(\R,t_{1},h_{1})$  as is $t/H \in V$. Now $h_{1}\cdot (1/H) = 1/H$, hence clearly $h_{1}\cdot (t/H)$ is also infinitesimally close (over $M,d_{1},h_{1}$)  to $1/H$.

\begin{Claim}
$h_{1}\cdot (t/H)$ is on the ``positive" side of $1/H$ and realizes the unique heir of $\pi(q_{0})$ over $\R,t_{1},h_{1}$. 
\end{Claim}
\begin{proof}[Proof of Claim]
When we mention ``positive side" we are identifying $V$ and $T$. Now the map $\pi_{1}$ from $T$ to $\Pp^{1}$  is a ``local homeomorphism" taking the identity to $\left(\begin{array}{c} 1\\ 0    \end{array}\right)$ and taking positive infinitesimals in $T$ to  infinite $x\in K$ and (by definition) respects the action of $G$. Hence it suffices to show that for $h_{1} = (b,c)\in H$, and $x\in K$ infinite, such that $tp(h_{1}/\R,x)$ is finitely satisfiable in $\R$, then $h_{1}\cdot x  =  b^{2}x + bc$  is (positive) infinite over $\R,b,c$.  This is clear: Firstly, $x$ is infinite over $\R,b,c$. Now as $b^{2} > 0$, $b^{2}x$ is positive infinite over $\R,b,c$,  as is $b^{2}x+ bc$.
\end{proof}

By the claim $\tp(h_{1}\cdot (t/H)/\R, t_{1},h_{1}) = \tp((t/H)/\R, t_{1}, h_{1})$. So without loss of generality $h_{1} =1$.
So we are in the situation of $t,t_{1}\in T$, $t$ realizes $q_{0}$ and $\tp(t_{1}/\R,t)$ is finitely satisfiable over $\R$. It is then clear that 
$t_{1}t$ realizes $q_{0}$ if and only if  $t_{1}$ is the identity, or itself realizes $q_{0}$. As $t_{1}\cdot (t/H) = (t_{1}t/H)$, and by virtue of $\pi$ inducing a homeomorphism between $S_{T}(\R)$ and $S_{V}(\R)$, we see that $t_{1}\cdot(t/H)$ realizes $\pi(q_{0})$ if and only if $t_{1}$ is the identity or a realization of $q_{0}$. This proves the lemma.
\end{proof}

\begin{Corollary} \label{cor:q} Let $t\in T$ realize $q_{0}$ and let $h_{1}\in H$ be such that $\tp(h_{1}/\R,t)$ is finitely satisfiable in $\R$. Then $h_{1}t = t_{1}h_{2}$ for $t_{1}\in T$ realizing $q_{0}$. 
\end{Corollary}
\begin{proof} We have just seen in the first part of the proof of \ref{lem:stab} that $h_{1}\cdot (t/H)$ realizes $\pi(q_{0})$, which suffices.
\end{proof}

\begin{Lemma} \label{lem:q}  For any $p\in S_{G}(\R)$, $p*r_{0} = r_{0}$ if and only if $p = \tp(t_{1}h_{1}/\R)$ with $h_{1}\in H$ and $t_{1}\in T$ realizing $q_{0}$.
\end{Lemma} 
\begin{proof} If $p*r_{0} = r_{0}$ then by Lemma \ref{lem:p}, $p*\pi(r_{0}) = \pi(r_{0})$. As $\pi(r_{0}) = \pi(q_{0})$, by Lemma \ref{lem:stab}, $p$ is of the required form.

Now let $p = \tp(t_{1}h_{1}/M)$ with $t_{1}$ realizing $q_{0}$. Suppose  $th$ realizes $r_{0}$ and $\tp(t_{1},h_{1}/\R,t,h)$ is finitely satisfiable in $\R$. Note that (as $\tp(t/M,h)$ is finitely satisfiable in $\R$) we have that $\tp(t_{1},h_{1},t/M,h)$ is finitely satisfiable in $\R$, so by Lemma \ref{lem:heir}$(iii)$,  $\tp(h/M,t_{1},h_{1},t)$ realizes $p_{0}|(M,t_{1},h_{1},t)$. Now by Corollary \ref{cor:q}, $h_{1}t = t_{2}h_{2}$ for $t_{2}$ realizing $q_{0}$ and $h_{2}\in H$. Note that we still have $\tp(h/M,t_{1},t_{2},h_{2}) = p_{0}|(M,t_{1},t_{2},h_{2})$. Now $p_{0}$ is a (definable) left $H(\R)$-invariant type of $H$,  so for any  model $K'\supset \R$, $p_{0}|K'$ is also a left $H(K')$-invariant type of $H$. Hence $h_{3} = h_{2}h$ realizes $p_{0}|(\R,t_{1},t_{2},h_{2})$.

Now $t_{1}h_{1}th = t_{1}t_{2}h_{3}$. As $t_{1},t_{2}$ both realize $q_{0}$ so does their product $t_{1}t_{2}$ and we have just seen that $\tp(t_{1}t_{2}/\R,h_{3})$ is the unique coheir over $(\R,h_{3})$ of $q_{0}$. So  $t_{1}t_{2}h_{3}$ realizes $r_{0}$ as required.
\end{proof} 

From Lemmas \ref{lem:stab} and \ref{lem:q} we conclude easily:

\begin{Corollary} \label{cor:homeo} The restriction of $\pi\colon S_{G}(\R) \to S_{V}(\R)$ to $\cl(G(\R)r_{0})$ is a homeomorphism betweeen $\cl(G(\R)r_{0})$ and $\cl(G(\R)\cdot\pi(r_{0}))$
\end{Corollary}

\begin{Lemma} \label{lem:uniq} The set $S_{V,na}(\R)$ of nonalgebraic types in $S_{V}(\R)$ is the unique minimal closed $G(\R)$-invariant subset of $S_{V}(\R)$.   
\end{Lemma}
\begin{proof} Let for now $S$ denote the set of nonalgebraic types in $S_{V}(\R)$, a closed subspace. It is obviously $G(\R)$-invariant. To show minimality it is enough, to note that, identifying (via $\pi$) $S$ with the space of nonalgebraic types in $S_{T}(\R)$, it is  minimal closed $T(\R)$-invariant. If $U$ is a basic open subset of $S_{T}(\R)$ which is not a finite set (of isolated points), then by $o$-minimality, $U$ contains an ``interval" in $S_{T}(\R)$, namely the set of types containing a formula defining an interval, with endpoints in $T(\R)$, in $T$ with respect to the circular ordering on $T$. But clearly if $s\in S$ then for some $g\in T(\R)$, $gs\in U$. So $S$ contains no proper $T(\R)$-invariant closed subset, whereby $S$ is minimal closed $T(\R)$-invariant, as required. Uniqueness is clear.  
\end{proof}

From \ref{cor:homeo} and \ref{lem:uniq} we deduce:

\begin{Proposition} \label{prop:min} 
$I =  \cl(G(\R)r_{0})$ is a minimal $G(\R)$-invariant closed subspace of $S_{G}(\R)$ and is homeomorphic as a $G(\R)$-flow to  $S_{V,na}(\R)$ under $\pi$.
\end{Proposition} 

\begin{Lemma} \label{lem:idem}  $r_{0}*r_{0} = r_{0}$.  Namely $r_{0}$ is an idempotent in $I$. 

\end{Lemma}
\begin{proof} This is special case of Lemma \ref{lem:q}. 
\end{proof}

So we have so far accomplished the first aim: description of a minimal (closed) subflow $I$ of $S_{G}(\R)$ and an idempotent $r_{0}\in I$. 
We now want to describe the ``ideal group" $r_{0}*I$.  Note first:

\begin{Lemma} \label{lem:bij} The restriction of $\pi$ to $r_{0}*I$ is a bijection with $r_{0}*S_{V,na}(M)$

\end{Lemma}
\begin{proof} By \ref{lem:p} and \ref{prop:min}.
\end{proof}

We first consider the action of $H<G$ on $\Pp^{1}$ from Section \ref{sec:sl}. Note that $H$ fixes $ \left(\begin{array}{c} 1\\ 0    \end{array}\right)$.  We identify any other element  $\left(\begin{array}{c} x\\ 1    \end{array}\right)$ of $\Pp^{1}$ with $x\in K$.  With this notation:

\begin{Lemma} \label{lem:inf} Let $x\in \Pp^{1}$, let $h$ realize $p_{0}$ such that $\tp(h/\R,x)$ is finitely satisfiable in $\R$. Then $h\cdot x$ is positive infinite or negative infinite  (in particular infinitesimally close to $\left(\begin{array}{c} 1\\ 0    \end{array}\right)$ in $\Pp^{1}$).   
\end{Lemma} 
\begin{proof}  So $h = (b,c)$  with $b$ (positive) infinite, and $c$ (positive) infinite over $b$. And $h\cdot x = b^{2}x + bc$.
\begin{itemize}
\item Case (i), $x$ is finite  (positive or negative). Then clearly $b^{2}{x} + bc$ is positive infinite  (as $bc$ is infinite over $|b^{2}x|$).
\item Case (ii),  $x$ is positive infinite. Then clearly $b^{2}x + bc$ is positive infinite.
\item Case (iii), $x$ is negative infinite. Now as $\tp(x/b,c)$ is definable over $\R$, $x$ is negative infinite \emph{over}  $\{b,c\}$, i.e. $x <  \dcl(b,c)$. Hence $b^{2}x < \dcl(b,c)$, whereby  $b^{2} + bc$ is still negative infinite.
\end{itemize}
\end{proof}

Now we consider the homeomorphism (induced by $\pi_{1}$) between $T/\{I,-I\}$ and $\Pp^{1}$ given in \ref{rem:homeo} and the corresponding action $\cdot_{1}$ of $G$ on $T/\{I,-I\}$.   As the identity of $T/\{I,-I\}$ goes to $\left(\begin{array}{c} 1\\ 0    \end{array}\right)$  under $\pi_{1}$, we deduce from Lemma \ref{lem:inf}:

\begin{Corollary} \label{cor:inf} Let $t/\{I,-I\} \in T/\{I,-I\}$ and let $h$ realize $p_{0}$ such that $\tp(h/\R,t)$ is finitely satisfiable in $\R$. Then $h\cdot_{1}(t/\{I,-I\})$ is infinitesimal in $T/\{I,-I\}$ (namely infinitesimally close to the identity or equal to the identity). 
\end{Corollary}

We now consider the action $\cdot$ of $G$ on $T$ induced by the action of $G$ on $V$ and the homeomorphism between $T$ and $V$ induced by $\pi$. And we use Remark \ref{rem:act} to conclude:

\begin{Lemma} \label{lem:inff} Let $t\in T$ and let $h$ realize $p_{0}$ such that $\tp(h/\R,t)$ is finitely satisfiable in $\R$. Then $h\cdot t$ is infinitesimally close to the identity $I$  (i.e. $(1,0)$)  or to $-I$  (i.e. $(-1,0)$). Moreover both possibilities happen. Namely if  $t$ is infinitesimally close to $-I$ so is $h\cdot t$, and if $t$ is infinitesimally close to $I$ then so is $h\cdot t$. 
\end{Lemma}
\begin{proof} The first part follows from Corollary \ref{cor:inf}. The rest follows by continuity and the fact that $h\cdot I = I$ and $h\cdot -I = -I$  (as  $-I$ commutes with $h$).
\end{proof}

Remember $q_{0}$ is the type of a ``positive infinitesimal" in $T$. We let $q_{1}$ denote the type of an element of $T$ infinitesimally close to $-I$ and on the ``positive side".
Now we can conclude:

\begin{Proposition} \label{prop:main} $r_{0}*S_{V,na}(\R)$ has two elements, $\pi(q_{0})$ and $\pi(q_{1})$. 
\end{Proposition}
\begin{proof}  We will work instead with the action $\cdot$ of $G$ on $S_{T}(\R)$ induced by the homeomorphism induced by $\pi$ between $S_{T}(\R)$ and $S_{V}(\R)$.  
So for  a type $q$ of an element of $T$, by $r_{0}*q$ we mean $\tp(g\cdot t/\R)$ where $t$ realizes $q$ and $g$ realizes $r_{0}$ such that $\tp(g/\R,t)$ is finitely satisfiable in $\R$. 

So let  $t_{1} \in T\setminus T(\R)$  (i.e. $t_{1}$ realizes a nonalgebraic type in $S_{T}(\R)$). And let $th$ realize $r_{0}$ such that $\tp(th/\R,t_{1})$ is finitely satisfiable in $\R$. Then $\tp(h/\R,t_{1})$ is finitely satisfiable in $\R$, and we may assume that $\tp(t/\R,h,t_{1})$ is finitely satisfiable in $\R$. (And remember that $t$ realizes $q_{0}$ and $h$ realizes $p_{0}$).  By Lemma \ref{lem:inff}, $h\cdot t_{1} = t_{2}$ say,  is infinitesimally close to either $I$ or $-I$  (and each can happen for suitable choice of $t_{1}$). Note also that $t\cdot t_{2}$ is just $tt_{2}$ (product in $T$). Now  as $t$ realizes $q_{0}$ and its type over $(\R,t_{2})$ is finitely satisfiable in $\R$ it is easy to see that $tt_{2}$ realizes  $q_{0}$ if $t_{2}$ is infinitesimally close to $I$, and realizes $q_{1}$ if $t_{2}$ is infinitesimally close to $-I$.  This concludes the proof of Proposition \ref{prop:main}.
\end{proof}

Putting together with earlier results we summarize (where $r_{0}$ is as in Definition \ref{def:type}).

\begin{Theorem} \label{thm:main}
\begin{enumerate}
\item[(i)] $I = \cl(G(\R)r_{0})$ is a minimal closed $G(\R)$-invariant subset of $S_{G}(M)$.
\item[(ii)] $r_{0}$ is an idempotent, with respect to the Ellis semigroup structure $*$ on $S_{G}(\R)$.
\item[(iii)] The ideal group $(r_{0}*I, *)$ has two elements.
\end{enumerate}
\end{Theorem}
\begin{proof} (i) is Proposition \ref{prop:min}. (ii) is Lemma \ref{lem:idem}. And (iii) follows from Proposition \ref{prop:main} and Lemma \ref{lem:bij}.
\end{proof}

\vspace{5mm}
\noindent
We finish the paper with some remarks on extensions of our results.  Note the ideal group above is precisely the centre of 
$\SL(2,\R)$. Now any finite cover $G$ of $\PSL(2,\R)$ can be realized as a semialgebraic group over $\R$. The proofs above go over, essentially word-for-word, to show that the ideal group of $G$ coincides with the finite group $Z(G)$. The idempotent $r_{0}$ and minimal closed $G$-invariant subset $I$ of $S_{G}(\R)$ are chosen in exactly the same way, and $Z(G)$ is contained in $T$.\\

It is natural to also ask about the case where $G = \widetilde{\SL(2,\R)}$, the universal cover of $\SL(2,\R)$. Now $G$ can be naturally interpreted (defined) in the two-sorted structure $M = ((\Z,+), (\R,+,\times))$. Again all types over this standard model are definable, the $\Z$-sort being stable. 
$H$ will be as before and the role of the maximal compact $T$ is now played by the universal cover of $SO(2,\R)$, interpreted naturally on the set
$\Z \times SO(2,\R)$. The above analysis goes through to show, among other things, that the ideal group is $\widehat{\Z}$, the profinite completion of 
$(\Z,+)$, which is in fact precisely the set of generic types of $\Z$. \\

We expect that all this again goes through, with an arbitrary real closed field $K$ in place of $\R$, where we expand the structure by all externally definable sets and consider the action of $G(K)$ on $S_{ext,G}(K)$. \\

Finally we expect that essentially the same results hold for arbitrary simple (modulo discrete centre) noncompact Lie groups. 



\end{document}